\pdfoutput=1
\documentclass[a4paper,reqno]{amsart}
\usepackage[margin=3cm]{geometry}

\newcommand{\ifarticle}[2]{
    \csname@ifclassloaded\endcsname{beamer}{#2}{#1}
}

    \usepackage{xparse} 
    \usepackage{xspace} 
    \usepackage{etoolbox} 
    \usepackage{xpatch} 
    \usepackage{pgffor} 

    \usepackage{amsmath,amssymb,amsthm} 
    \allowdisplaybreaks[1]

    \usepackage{mathtools} 
    \usepackage{mathrsfs} 
    \usepackage{stmaryrd} 
    \usepackage{bbm} 
    \usepackage{quiver} 
    \usepackage{xcolor} 
    \usepackage{scalerel} 
    \usepackage{booktabs} 
    \usepackage{nameref} 
    \usepackage[british]{babel} 
    \usepackage{csquotes} 
    \usepackage[UKenglish]{isodate} 
    \usepackage{microtype} 
    \usepackage[splitrule]{footmisc} 

    \ifarticle{
        \PassOptionsToPackage{hyphens}{url}
        \usepackage[pdfusetitle,linktoc=all,colorlinks,citecolor=cyan,linkcolor=purple,urlcolor=blue]{hyperref} 
        \urlstyle{rm}

        \usepackage[nameinlink,noabbrev,capitalize]{cleveref} 
        \usepackage{footnotebackref} 

        \usepackage[shortlabels]{enumitem} 
        \setlist{topsep=2pt,itemsep=2pt,partopsep=2pt,parsep=2pt} 
    }{}
    \usepackage[style=alphabetic,backref=true,backrefstyle=three,maxnames=9,minalphanames=3]{biblatex} 
    \DeclareUnicodeCharacter{0301}{\TODO[Invalid symbol.]} 

    \DeclareFieldFormat*{citetitle}{\emph{#1}} 
    \DeclareCiteCommand{\citetitle}
        {\boolfalse{citetracker}%
        \boolfalse{pagetracker}%
        \usebibmacro{prenote}}
        {\ifciteindex
            {\indexfield{indextitle}}
            {}%
        \printtext[bibhyperref]{\printfield[citetitle]{labeltitle}}}
        {\multicitedelim}
        {\usebibmacro{postnote}}
    \DeclareCiteCommand*{\citeauthor}
        {\defcounter{maxnames}{99}%
        \defcounter{minnames}{99}%
        \defcounter{uniquename}{2}%
        \boolfalse{citetracker}%
        \boolfalse{pagetracker}%
        \usebibmacro{prenote}}
        {\ifciteindex{\indexnames{labelname}}{}%
        \printnames{labelname}}
        {\multicitedelim}
        {\usebibmacro{postnote}}

    \usepackage{calc} 
    \usepackage{tabularray} 
    \usepackage{ebproof} 
    \usepackage{forest} 
    \usepackage{adjustbox} 
    \usepackage{float} 
    \usepackage{stfloats}

    \makeatletter
    \ifarticle{
        \xpretocmd{\@adminfootnotes}{\let\@makefntext\BHFN@OldMakefntext}{}{}
        \renewcommand\@makefntext[1]{%
        \@ifundefined{@makefnmark}
            {}
            {%
            \renewcommand\@makefnmark{%
            \mbox{%
                \textsuperscript{%
                \normalfont
                \hyperref[\BackrefFootnoteTag]{\@thefnmark}%
                }%
            }\,%
            }%
            \BHFN@OldMakefntext{#1}%
        }%
        }

        \LetLtxMacro{\BHFN@Old@footnotemark}{\@footnotemark}
        \renewcommand*{\@footnotemark}{%
            \refstepcounter{BackrefHyperFootnoteCounter}%
            \xdef\BackrefFootnoteTag{bhfn:\theBackrefHyperFootnoteCounter}%
            \label{\BackrefFootnoteTag}%
            \BHFN@Old@footnotemark
        }

        \makeatletter
        \def\paragraph{\@startsection{paragraph}{4}%
          \z@\z@{-\fontdimen2\font}%
          {\normalfont\bfseries}}
        \makeatother
    }{}
    \makeatother

    \ifarticle{
        \theoremstyle{plain}
        \newtheorem{theorem}{Theorem}[section]
        \newtheorem{proposition}[theorem]{Proposition}
        \newtheorem{lemma}[theorem]{Lemma}
        \newtheorem{corollary}[theorem]{Corollary}

        \newtheorem*{theorem*}{Theorem}
        \newtheorem*{corollary*}{Corollary}

        \theoremstyle{definition}
        \newtheorem{definition}[theorem]{Definition}
        
        \newtheorem{example}[theorem]{Example}

        \newtheorem{remark}[theorem]{Remark}

        \newenvironment{sketch}{\proof}{\endproof}

        \Crefname{theoremenumi}{Theorem}{Theorems}
        \AtBeginEnvironment{theorem}{%
            \crefalias{enumi}{theoremenumi}%
            \setlist[enumerate,1]{
                ref={\csname thetheorem\endcsname.(\arabic*)}
            }%
            \crefalias{enumii}{theoremenumi}%
            \setlist[enumerate,2]{
                ref={\thetheorem.(\arabic*).(\alph*)}
            }%
        }

        \Crefname{propositionenumi}{Proposition}{Propositions}
        \AtBeginEnvironment{proposition}{%
            \crefalias{enumi}{propositionenumi}%
            \setlist[enumerate,1]{
                ref={\csname theproposition\endcsname.(\arabic*)}
            }%
            \crefalias{enumii}{propositionenumi}%
            \setlist[enumerate,2]{
                ref={\theproposition.(\arabic*).(\alph*)}
            }%
        }

        \Crefname{lemmaenumi}{Lemma}{Lemmas}
        \AtBeginEnvironment{lemma}{%
            \crefalias{enumi}{lemmaenumi}%
            \setlist[enumerate,1]{
                ref={\csname thelemma\endcsname.(\arabic*)}
            }%
            \crefalias{enumii}{lemmaenumi}%
            \setlist[enumerate,2]{
                ref={\thelemma.(\arabic*).(\alph*)}
            }%
        }

        \Crefname{corollaryenumi}{Corollary}{Corollaries}
        \AtBeginEnvironment{corollary}{%
            \crefalias{enumi}{corollaryenumi}%
            \setlist[enumerate,1]{
                ref={\csname thecorollary\endcsname.(\arabic*)}
            }%
            \crefalias{enumii}{corollaryenumi}%
            \setlist[enumerate,2]{
                ref={\thecorollary.(\arabic*).(\alph*)}
            }%
        }

        \Crefname{definitionenumi}{Definition}{Definitions}
        \AtBeginEnvironment{definition}{%
            \crefalias{enumi}{definitionenumi}%
            \setlist[enumerate,1]{
                ref={\csname thedefinition\endcsname.(\arabic*)}
            }%
            \crefalias{enumii}{definitionenumi}%
            \setlist[enumerate,2]{
                ref={\thedefinition.(\arabic*).(\alph*)}
            }%
        }

        \Crefname{exampleenumi}{Example}{Examples}
        \AtBeginEnvironment{example}{%
            \crefalias{enumi}{exampleenumi}%
            \setlist[enumerate,1]{
                ref={\csname theexample\endcsname.(\arabic*)}
            }%
            \crefalias{enumii}{exampleenumi}%
            \setlist[enumerate,2]{
                ref={\theexample.(\arabic*).(\alph*)}
            }%
        }

        \foreach \env in {definition,remark,example}{
        \AtBeginEnvironment{\env}{%
          \pushQED{\qed}%
        }
        \AtEndEnvironment{\env}{\popQED\endexample}
        }
    }{}

    \DeclareDocumentCommand{\mathcommand}{mO{0}m}{%
    \expandafter\let\csname old\string#1\endcsname=#1
    \expandafter\newcommand\csname new\string#1\endcsname[#2]{#3}
    \DeclareRobustCommand#1{%
        \ifmmode
        \expandafter\let\expandafter\next\csname new\string#1\endcsname
        \else
        \expandafter\let\expandafter\next\csname old\string#1\endcsname
        \fi
        \next
    }%
    }

    \mathcommand{\h}{\textup{-}}

    \newcommand{\tx}{\mathrm}
    \mathcommand{\b}{\mathbf}
    
    \newcommand{\cl}{\mathcal}
    \mathcommand{\bb}{\mathbb}
    \DeclareMathAlphabet{\bbn}{U}{bbold}{m}{n}
    
    \newcommand{\scr}{\mathscr}
    \mathcommand{\sf}{\mathsf}
    \mathcommand{\u}{\underline}

    \newcommand{\TODO}[1][TODO]{\textcolor{orange}{\textup{#1}}\xspace}

    \newcommand{\datetoday}{\date{\cleanlookdateon\today}}

    \newcommand{\defeq}{\mathrel{:=}}
    \mathcommand{\d}{\mathbin{;}}
    \mathcommand{\c}{\circ}
    \newcommand{\ph}[1][]{{({-}_{#1})}}
    
    \newcommand{\iso}{\cong}
    \renewcommand{\equiv}{\simeq}
    \newcommand{\biequiv}{\sim}

    \newcommand{\tto}{\Rightarrow}
    
    \newcommand{\xtto}{\xRightarrow}
    
    \newcommand{\ffto}{\hookrightarrow}
    
    \newcommand{\epito}{\twoheadrightarrow}
    
    \makeatletter
    \def\slashedarrowfill@#1#2#3#4#5{%
    $\m@th\thickmuskip0mu\medmuskip\thickmuskip\thinmuskip\thickmuskip
    \relax#5#1\mkern-7mu%
    \cleaders\hbox{$#5\mkern-2mu#2\mkern-2mu$}\hfill
    \mathclap{#3}\mathclap{#2}%
    \cleaders\hbox{$#5\mkern-2mu#2\mkern-2mu$}\hfill
    \mkern-7mu#4$%
    }
    \def\rightslashedarrowfill@{%
    \slashedarrowfill@\relbar\relbar\mapstochar\rightarrow}
    \newcommand\xslashedrightarrow[2][]{%
    \ext@arrow 0055{\rightslashedarrowfill@}{#1}{#2}}
    \def\leftslashedarrowfill@{%
    \slashedarrowfill@\leftarrow\relbar\mapsfromchar\relbar}
    \newcommand\xslashedleftarrow[2][]{%
    \ext@arrow 0055{\leftslashedarrowfill@}{#1}{#2}}
    \makeatother

    \newcommand{\inv}{^{-1}}
    \newcommand{\op}{{}^\tx{op}}

    \newcommand{\lx}{\mathbin{\rhd}}
    
    \newcommand{\adj}{\dashv}
    
    \DeclareFontFamily{U}{min}{}
    \DeclareFontShape{U}{min}{m}{n}{<-> udmj30}{}

    \mathcommand{\comma}{\downarrow}

    \newsavebox{\whitecircstar}\sbox{\whitecircstar}{\kern.075em\tikz{\node[draw, circle,line width=.36pt, inner sep=0]{$*$};}\kern.075em}
    
    \newsavebox{\blackcircstar}\sbox{\blackcircstar}{\kern.075em\tikz{\node[fill, circle, line width=.36pt, inner sep=0, text=white]{$*$};}\kern.075em}

    \makeatletter
    \def\widebreve{\mathpalette\wide@breve}
    \def\wide@breve#1#2{\sbox\z@{$#1#2$}%
         \mathop{\vbox{\m@th\ialign{##\crcr
    \kern0.08em\brevefill#1{0.8\wd\z@}\crcr\noalign{\nointerlineskip}%
                        $\hss#1#2\hss$\crcr}}}\limits}
    \def\brevefill#1#2{$\m@th\sbox\tw@{$#1($}%
      \hss\resizebox{#2}{\wd\tw@}{\rotatebox[origin=c]{90}{\upshape(}}\hss$}
    \makeatother

    \NewDocumentCommand{\jrule}{om}{%
        \IfNoValueTF{#1}
            {\textsc{#2}}
            {$#1$-\textsc{#2}}%
    }

    \newcommand{\Set}{{\b{Set}}}
    
    \newcommand{\V}{{\bb V}} 
    
    \newcommand{\CAT}{\b{CAT}}
    
    \newcommand{\VCAT}{\V\h\CAT}
    
    \newcommand{\Alg}{\b{Alg}}
    
    \newcommand{\ff}{fully faithful}
    
    \newcommand{\ffness}{full faithfulness}
    
    \newcommand{\lff}{locally \ff{}}

    \newcommand{\ie}{i.e.\@\xspace}
    
    \newcommand{\cf}{cf.\@\xspace}

    \NewDocumentCommand{\etc}{t.}{etc.\@\xspace}
    \NewDocumentCommand{\ibid}{t.}{ibid.\@\xspace}
    \NewDocumentCommand{\loccit}{t.}{loc.\ cit.\@\xspace}

\makeatletter

\define@key{beamerframe}{c}[true]{
    \beamer@frametopskip=0pt plus 1fill\relax%
    \beamer@framebottomskip=0pt plus 1fill\relax%
}

\patchcmd{\beamer@sectionintoc}{\vfill}{\vskip\itemsep}{}{}

\makeatother

\ifarticle{}{

  \colorlet{colour-bg}{black!85} 
  \definecolor{colour-primary}{HTML}{cc80ff} 
  \colorlet{colour-text}{black!10} 
  \colorlet{colour-subtle}{black!40} 
  \colorlet{colour-block-bg}{black!80} 
  \definecolor{colour-warning-bg}{HTML}{ffea80} 
  \definecolor{colour-warning-primary}{HTML}{e08152} 

  \hypersetup{linktoc=all,colorlinks, citecolor={colour-primary}, linkcolor={colour-primary}, urlcolor={colour-primary}} 
  \urlstyle{sf}

  \usepackage{changepage} 

  \usepackage{ragged2e} 

  \setbeamertemplate{section in toc}[sections numbered]

  \apptocmd{\frame}{}{\justifying}{}

  \usefonttheme[onlymath]{serif}

  \beamertemplatenavigationsymbolsempty
  \setbeamertemplate{footline}{
    \hfill%
    \usebeamercolor[fg]{page number in head/foot}%
    \usebeamerfont{page number in head/foot}%
    \setbeamertemplate{page number in head/foot}[pagenumber]%
    \usebeamertemplate*{page number in head/foot}\kern.2cm\vskip.2cm%
  }

  \setbeamertemplate{frametitle}[default][center]
  \addtobeamertemplate{frametitle}{\vspace*{1ex}}{\vspace*{1ex}}

  \setbeamertemplate{itemize item}{$\bullet$}

  \setbeamertemplate{theorems}[numbered]
  \newtheorem{proposition}[theorem]{\translate{Proposition}}

  \addtobeamertemplate{block begin}
  {}
  {\vspace{0ex} 
  \begin{adjustwidth}{1em}{1em} 
  \tikzcdset{background color=colour-block-bg}
  }
  \addtobeamertemplate{block end}{\end{adjustwidth}%
  \vspace{1ex}} 
  {}
  \renewenvironment<>{block}[1]{%
      \begin{actionenv}#2%
        \def\insertblocktitle{\begin{adjustwidth}{.8em}{.8em}#1\end{adjustwidth}}%
        \par%
        \usebeamertemplate{block begin}}
      {\par%
        \usebeamertemplate{block end}%
      \end{actionenv}}

  \addtobeamertemplate{block example begin}
  {}
  {\vspace{0ex} 
  \begin{adjustwidth}{1em}{1em} 
  \tikzcdset{background color=colour-block-bg}
  }
  \addtobeamertemplate{block example end}{\end{adjustwidth}%
  \vspace{1ex}} 
  {}
  \renewenvironment<>{exampleblock}[1]{%
      \begin{actionenv}#2%
          \def\insertblocktitle{\begin{adjustwidth}{.8em}{.8em}#1\end{adjustwidth}}%
          \par%
          \only<presentation>{
            \setbeamercolor{local structure}{parent=example text}}%
          \usebeamertemplate{block example begin}}
        {\par%
          \usebeamertemplate{block example end}%
        \end{actionenv}}

    \setbeamercolor{background canvas}{bg=colour-bg}
    \tikzcdset{background color=colour-bg}

    \setbeamercolor{block title}{fg=colour-primary,bg=colour-block-bg}
    \setbeamercolor{block title example}{fg=colour-primary,bg=colour-block-bg}
    \setbeamercolor{block body}{bg=colour-block-bg}
    \setbeamercolor{block body example}{bg=colour-block-bg}

    \setbeamercolor{title}{fg=colour-primary}
    \setbeamercolor{frametitle}{fg=colour-primary}
    \setbeamercolor{alerted text}{fg=colour-primary}

    \setbeamercolor{normal text}{fg=colour-text}
    \setbeamercolor{item}{fg=colour-text}
    \setbeamercolor{section in toc}{fg=colour-text}

    \setbeamercolor{page number in head/foot}{fg=colour-subtle}

    \setbeamercolor*{bibliography entry author}{fg=colour-text}
    \setbeamercolor*{bibliography entry note}{fg=colour-text}

}

\bibliography{references}

\newcommand{\K}{{\cl K}}
\newcommand{\I}{{\widetilde I}}
\renewcommand{\V}{{\scr V}}

\title[Adjoint functor theorems for lax-idempotent pseudomonads]{Adjoint functor theorems\\for lax-idempotent pseudomonads}
\author{Nathanael Arkor}
\address{Department of Mathematics and Statistics, Faculty of Science, Masaryk University, Czech Republic}
\author{Ivan Di Liberti}
\address{Department of Mathematics, Stockholm University, Stockholm, Sweden}
\author{Fosco Loregian}
\address{Tallinn University of Technology, Tallinn, Estonia}

\datetoday

\hypersetup{pdfauthor={Nathanael Arkor, Ivan Di Liberti, Fosco Loregian}}

\keywords{Adjoint functor theorem, relative adjunction, lax-idempotent pseudomonad, KZ-doctrine, free cocompletion, pseudodistributive law, 2-category, formal category theory}
\subjclass[2020]{18D70,18D65,18C15,18A35,18A40,18D20,18N10}

\begin{document}

\begin{abstract}
    For each pair of lax-idempotent pseudomonads $R$ and $I$, for which $I$ is locally \ff{} and $R$ distributes over $I$, we establish an adjoint functor theorem, relating $R$-cocontinuity to adjointness relative to $I$. This provides a new perspective on the nature of adjoint functor theorems, which may be seen as methods to decompose adjointness into cocontinuity and relative adjointness. As special cases, we recover variants of the adjoint functor theorem of Freyd, the multiadjoint functor theorem of Diers, and the pluriadjoint functor theorem of Solian--Viswanathan, as well as the adjoint functor theorems for locally presentable categories. More generally, we recover enriched $\Phi$-adjoint functor theorems for weakly sound classes of weight~$\Phi$.
\end{abstract}

\maketitle

\section{Introduction}

A fundamental result in category theory asserts that left-adjoint functors are cocontinuous, \ie{} that they preserve colimits. The role of an \emph{adjoint functor theorem} is to provide sufficient conditions for the converse to hold: that is, given a cocontinuous functor $f$, to provide sufficient conditions for $f$ to be left-adjoint. This is fundamentally a question of size: right adjoint functors may be computed as left extensions, hence as certain -- potentially large -- colimits. Therefore, adjoint functor theorems typically take the form of imposing a size constraint on $f$, limiting the size of this colimit, and hence ensuring that it exists.

In this paper, we are concerned with understanding the precise nature of the interplay between adjointness, cocontinuity, and size. It is our thesis that adjoint functor theorems -- and, more generally, relative adjointness theorems -- may be seen as an inherently 2-categorical phenomenon, arising from the interaction between a pair of lax-idempotent pseudomonads on a 2-category. To justify our claim, we shall prove a general adjoint 1-cell theorem in a 2-category (\cref{DAFT}), and show that it captures many instances of adjoint functor theorem present in the literature (\cref{adjoint-iff-cocontinuous-plus-admissible,AFTs-for-weights,relative-AFT,presentable-AFT}).

There are two distinct aspects to our work.
\begin{enumerate}
    \item From a purely formal, 2-categorical perspective, our result is an extension of the work of \textcite{walker2019distributive} on the interaction between \emph{admissibility} for lax-idempotent pseudomonads, in the sense of \textcite{bunge1999bicomma}, and pseudodistributive laws. While \citeauthor{walker2019distributive} studies the setting in which a lax-idempotent pseudomonad $I$ lifts to the 2-category of pseudoalgebras for an \emph{arbitrary} pseudomonad, we are interested in the case where both pseudomonads are lax-idempotent. Under such an assumption, it is possible to sharply characterise the property of being admissible with respect to $I$ (\cref{DAFT}).
    \item From the perspective of ordinary and enriched category theory, our result ties together various notions of adjoint functor theorem, providing a conceptual understanding that sheds new light on a classical result. In particular, we explain how size constraints may be seen as relative adjointness conditions~\cite{ulmer1968properties}, and that adjoint functor theorems thereby decompose adjointness properties into relative adjointness and cocontinuity properties (\cref{overview-of-AFTs}).
\end{enumerate}

We shall begin by explaining precisely what we mean by this latter statement, before recalling the notion of lax-idempotent pseudomonad and explaining the fundamental role it plays in the nature of adjointness.

\subsection{A brief overview of adjoint functor theorems}
\label{overview-of-AFTs}

Recall that a category is \emph{small-cocomplete} if it admits all small colimits, \ie{} colimits indexed by small categories; and that a functor is \emph{small-cocontinuous} if it preserves all small colimits. We say that a functor $f \colon A \to B$ satisfies the \emph{solution set condition}\footnotemark{} if,
\footnotetext{With respect to Freyd's original terminology, who was concerned with establishing right adjointness rather than left adjointness, this would be called the \emph{co-solution set condition}.}%
for every object $b \in B$, the presheaf $B(f{-}, b) \colon A\op \to \Set$ is \emph{weakly multirepresentable}\footnotemark{}, \ie{} if there exists a small family of objects $(a_i)_{i \in I}$ in $A$ and an epimorphism ${\coprod_{i \in I} A({-}, a_i) \epito B(f{-}, b)}$ in $[A\op, \Set]$.%
\footnotetext{Such presheaves were called \emph{petty} by \citeauthor{freyd1968several}~\cite[\S I]{freyd1968several}; we prefer the compositional prefixes \emph{weak} of \textcite{kainen1971weak} and \emph{multi-} of \textcite{diers1980categories}.}%

The classical adjoint functor theorem, which appeared not long after the introduction of adjoint functors themselves~\cite{kan1958adjoint}, is stated as follows.
\begin{theorem*}[Freyd]
    Let $f$ be a functor from a locally small and small-cocomplete category. The following are equivalent.
    \begin{enumerate}
        \item $f$ is left-adjoint.
        \item $f$ satisfies the solution set condition and is small-cocontinuous.
    \end{enumerate}
\end{theorem*}
The theorem essentially originates in \cite[84]{freyd1964abelian}, though the general formulation appears, attributed to Freyd, as \cite[Theorem~V.6.2]{maclane1998categories}. Here, the size constraint takes the form of the solution set condition.

While the theorem of \citeauthor{freyd1964abelian} is perhaps the most well known, there is another that will be most appropriate for our purposes. Call a functor $f \colon A \to B$ \emph{small-admissible} if, for every object $b \in B$, the presheaf $B(f{-}, b) \colon A\op \to \Set$ is \emph{small}, \ie{} a small colimit of representables.

\begin{theorem*}[Ulmer]
    Let $f$ be a functor from a locally small and small-cocomplete category. The following are equivalent.
    \begin{enumerate}
        \item $f$ is left-adjoint.
        \item $f$ is small-admissible and small-cocontinuous.
    \end{enumerate}
\end{theorem*}

The theorem essentially originates in \cite[(13)]{ulmer1971adjoint}. Here, the size constraint takes the form of the small-admissibility condition.

The distinction between the size constraints in the two forms of the adjoint functor theorem stated above is slight. Every small-admissible functor satisfies the solution set condition, since every small colimit is a quotient of a small coproduct. Conversely, while every epimorphism in a presheaf category is regular, so that each weakly multirepresentable presheaf $q$ is the coequaliser of a pair of parallel morphisms $p \rightrightarrows \coprod_{i \in I} A({-}, a_i)$, it is not necessarily the case that $q$ is small unless $p$ is small. Therefore, a functor satisfying the solution set condition is not necessarily small-admissible. However, the two properties coincide under weak assumptions: in particular, a small-cocontinuous functor from a small-cocomplete category satisfies the solution set condition if and only if it is small-admissible~\cite[Lemma~12]{ulmer1971adjoint}. Consequently, for the purposes of the adjoint functor theorem, we may view the two conditions as being equivalent. Henceforth, we shall focus on \citeauthor{ulmer1971adjoint}'s formulation of the adjoint functor theorem, which we shall see admits a particularly elegant understanding.

We may understand the small-admissibility condition as a weakening of the notion of adjointness. First, observe that every functor $f \colon A \to B$ between locally small categories is left-adjoint \emph{relative} to the Yoneda embedding (in the sense of \cite[\S2]{ulmer1968properties}), since the Yoneda lemma exhibits an isomorphism $B(f a, b) \iso [A\op, \Set](A({-}, a), B(f{-}, b))$, natural in $a \in A$ and $b \in B$.
\[\begin{tikzcd}
	& B \\
	A && {[A\op, \Set]}
	\arrow[""{name=0, anchor=center, inner sep=0}, "f", from=2-1, to=1-2]
	\arrow[""{name=1, anchor=center, inner sep=0}, "{B(f{-}_2, {-}_1)}", from=1-2, to=2-3]
	\arrow["{A({-}_2, {-}_1)}"', hook, from=2-1, to=2-3]
	\arrow["\dashv"{anchor=center, rotate=2}, shift right=2, draw=none, from=0, to=1]
\end{tikzcd}\]
Thus, adjointness relative to the Yoneda embedding is a tautology. Conversely, adjointness relative to the identity functor on $A$ (equivalently to the embedding of $A$ into the full subcategory of $[A\op, \Set]$ spanned by the representable presheaves) is simply ordinary adjointness. Consequently, we may consider a spectrum of notions of adjointness, mediated by adjointness relative to the inclusion of $A$ into some full subcategory $\widetilde A$ of its category of presheaves $[A\op, \Set]$.
\[\begin{tikzcd}
	& B \\
	A && {\widetilde A}
	\arrow[""{name=0, anchor=center, inner sep=0}, "f", from=2-1, to=1-2]
	\arrow[""{name=1, anchor=center, inner sep=0}, "{B(f{-}_2, {-}_1)}", dashed, from=1-2, to=2-3]
	\arrow["{A({-}_2, {-}_1)}"', hook, from=2-1, to=2-3]
	\arrow["\dashv"{anchor=center}, shift right=2, draw=none, from=0, to=1]
\end{tikzcd}\]
At one end of the spectrum, where the notion of adjointness is strongest, we have $\widetilde A = A$, recovering ordinary adjointness. At the other end of the spectrum, where the notion of adjointness is weakest, we have $\widetilde A = [A\op, \Set]$, which is a triviality. Somewhere in-between, taking $\widetilde A = [A\op, \Set]_\tx{small}$ to be the full subcategory of small presheaves on $A$, we recover the definition of small-admissibility.\footnotemark{}
\footnotetext{Similarly, taking $\widetilde A$ to be the full subcategory of weakly multirepresentable presheaves on $A$, we recover the solution set condition (\cf{}~\cite[29]{lack2023virtual}).}

From this perspective, we may see the adjoint functor theorem as describing a way to decompose adjointness into two weaker properties -- relative adjointness and cocontinuity -- which may therefore be seen as being complementary to one another. It is consequently natural to wonder whether it might be possible to decompose adjointness in other ways, for instance, by strengthening the notion of relative adjointness, whilst simultaneously weakening the notion of cocontinuity.

To answer this question, we observe that both small-admissibility and small-cocontinuity are in some sense parameterised by a class of colimits: namely, the small colimits. For a fixed class of weights $\Psi$, small-admissibility and small-cocontinuity both admit natural generalisations from small colimits to $\Psi$-weighted colimits: \emph{$\Psi$-admissibility} is adjointness relative to the cocompletion $\widetilde A \defeq [A\op, \Set]_\Psi$ of $A$ under $\Psi$-weighted colimits; while $\Psi$-cocontinuity is the preservation of $\Psi$-weighted colimits. When $\Psi = \varnothing$, $\Psi$-admissibility is adjointness and $\Psi$-cocontinuity is trivial; whereas when $\Psi$ is the class of small weights, $\Psi$-admissibility is small-admissibility and $\Psi$-cocontinuity is small-cocontinuity. This might lead us to conjecture that it is possible to decompose adjointness not simply with respect to small colimits, but with respect to any class of weights. This is indeed the case.
\begin{theorem*}[Tholen]
    Let $f$ be a functor from a locally small and $\Psi$-cocomplete category. The following are equivalent.
    \begin{enumerate}
        \item $f$ is left-adjoint.
        \item $f$ is $\Psi$-admissible and $\Psi$-cocontinuous.
    \end{enumerate}
\end{theorem*}
The theorem essentially originates in \cite[Theorem~2.3]{tholen1984pro}. With regards to the discussion above, this theorem is conceptually enlightening, but it is also practical. For many classes of weights $\Psi$, notions of $\Psi$-adjointness have been studied in the literature: for instance, taking $\Psi$ to be the class of weights indexed by finite discrete categories, $\Psi$-admissibility is the \emph{multiadjointness} of \textcite{diers1977categories}; taking $\Psi$ to be the class indexed by finite categories, $\Psi$-admissibility is the \emph{pluriadjointness} of \textcite{solian1990pluri}; and taking $\Psi$ to be the class of absolute weights, $\Psi$-admissibility is the \emph{semiadjointness} of \textcite{hayashi1985adjunction} (\cf{}~\cite{hoofman1995remark}).

Given that adjointness is, itself, an admissibility property with respect to a class of weights $\Phi$ (namely, the empty class $\Phi = \varnothing$), we might ask whether it is possible to generalise \citeauthor{tholen1984pro}'s theorem further, and formulate a \emph{$\Phi$-adjoint functor theorem}: in other words, whether we might decompose \emph{adjointness with respect to $\Phi$} into a weaker adjointness property together with a cocontinuity property. This is indeed possible. Suppose that $\Psi$ and $\Phi$ are small classes of weights for which free $\Phi$-cocompletions preserve $\Psi$-cocompleteness. In this context, the following $\Phi$-adjoint functor theorem holds.
\begin{theorem*}[\cref{Phi-AFT}]
    Let $f$ be a functor between locally small and $\Psi$-cocomplete categories. The following are equivalent.
    \begin{enumerate}
        \item $f$ is $\Phi$-admissible.
        \item $f$ is $(\Phi \cup \Psi)$-admissible and $\Psi$-cocontinuous.
    \end{enumerate}
\end{theorem*}
In the presence of coaccessibility\footnotemark{}, and when $(\Phi \cup \Psi)$ is the class of all small weights, the $\Phi$-adjoint functor theorem above essentially appears as \cite[Proposition~3.17]{lack2023accessible} (\cf{}~\cref{AFTs-for-weights}). In addition, special cases of the theorem have been previously established in the literature: for instance, the adjoint functor theorems of \textcite{ulmer1971adjoint} and of \textcite{tholen1984pro}, the multiadjoint functor theorem of \textcite{diers1977categories}, and the pluriadjoint functor theorem of \textcite{solian1990pluri}.
However, the theorem does not appear to have been stated previously in complete generality; it will be a corollary of our main theorem (\cref{DAFT}).
\footnotetext{A category is coaccessible if it is the completion of a small category under cofiltered limits, \ie{} it is the dual of an accessible category.}

There is one further evident generalisation of the adjoint functor theorem. Namely, rather than simply considering adjoint functor theorems in the context of ordinary category theory, we may be interested in adjoint functor theorems in the context of enriched category theory, internal category theory, fibred and indexed category theory, and so on. The concept of adjunction makes sense not just in the setting of categories, functors, and natural transformations, but more generally in any 2-category~\cite{maranda1965formal}, and specialises to the appropriate notion of adjunction in each case. We might therefore hope that such adjoint functor theorems might be established in an arbitrary 2-category. However, to state and prove a general adjoint functor theorem at this level of generality, it is necessary to understand each of the concepts involved from the perspective of 2-category theory.

\subsection{Adjointness and lax-idempotent pseudomonads}

In fact, the $\Phi$-adjoint functor theorem, as we have presented it above, is immediately amenable to formalisation in the sense just described. It is necessary only to understand how to express the notions of \emph{cocompleteness}, \emph{cocontinuity}, and of \emph{admissibility} in a 2-categorical setting. Formalising these notions is precisely the role of a \emph{lax-idempotent pseudomonad}~\cite{kock1995monads,zoberlein1976doctrines,power2000representation,marmolejo1997doctrines,kelly1997property}.

A lax-idempotent pseudomonad is a 2-dimensional notion of monad that, conceptually, captures the notion of \emph{free cocompletion}. For instance, every class of small weights $\Psi$ induces a lax-idempotent pseudomonad $\overline\Psi$ on the 2-category $\CAT$ of locally small categories. Furthermore, the notions of cocompleteness, cocontinuity, and admissibility are captured as follows.
\begin{itemize}
    \item The $\overline\Psi$-pseudoalgebras are precisely the $\Psi$-cocomplete categories.
    \item The $\overline\Psi$-pseudomorphisms are precisely the $\Psi$-cocontinuous functors.
    \item The $\overline\Psi$-admissible 1-cells (\ie{} those 1-cells $f \colon A \to B$ for which $\overline\Psi f \colon \overline\Psi A \to \overline\Psi B$ is left-adjoint) are precisely the $\Psi$-admissible functors.
\end{itemize}
Furthermore, free $\Phi$-cocompletions preserve $\Psi$-cocompleteness if and only if there exists a pseudodistributive law of $\overline\Psi$ over $\overline\Phi$. Thus, a pseudodistributive law of lax-idempotent pseudomonads gives us precisely what we need to interpret the general adjoint functor theorem above in the setting of an arbitrary 2-category. This is our main theorem. We denote by $R$ a lax-idempotent pseudomonad taking the place of $\Psi$, and by $I$ a locally \ff{} lax-idempotent pseudomonad taking the place of $\Phi$. In the presence of a pseudodistributive law $RI \tto IR$, we have the following.
\begin{theorem*}[\cref{DAFT}]
    Let $f$ be a 1-cell between $R$-pseudoalgebras. The following are equivalent.
    \begin{enumerate}
        \item $f$ is $I$-admissible.
        \item $f$ is $IR$-admissible and $R$-cocontinuous.
    \end{enumerate}
\end{theorem*}
\begin{corollary*}[\cref{AFT}]
    Let $f$ be a 1-cell between $R$-pseudoalgebras. The following are equivalent.
    \begin{enumerate}
        \item $f$ is left-adjoint.
        \item $f$ is $R$-admissible and $R$-cocontinuous.
    \end{enumerate}
\end{corollary*}
In this way, the various adjoint functor theorems we have described may be seen as individual cases of a general relationship between admissibility and cocontinuity for lax-idempotent pseudomonads in the presence of a pseudodistributive law (\cref{AFTs-for-weights}).

\subsection*{Acknowledgements}

The authors thank Giacomo Tendas for clarifying discussions about cocompletions, pettiness, and (weak) adjoint functor theorems; and Gabriele Lobbia for helpful discussions regarding pseudomonads in Gray-categories. They thank also John Bourke, Dylan McDermott, Giacomo Tendas, Charles Walker, and the anonymous reviewer for helpful comments on the paper.

Ivan Di Liberti was supported by the Swedish Research Council (Vetenskapsrådet) under grant \textnumero{}\,2019-04545. The research has received funding from the Knut and Alice Wallenberg Foundation through its programme for mathematics.
Fosco Loregian was supported by the ESF-funded Estonian IT Academy research measure (project 2014-2020.4.05.19-0001).

\section{Pseudodistributive laws, admissibility, and adjoint functor theorems}

Our basic setting will be a 2-category $\K$ equipped with lax-idempotent pseudomonads $I$ and $R$, and a pseudodistributive law $RI \tto IR$. We shall recall the basic definitions, but assume familiarity with the general theory of lax-idempotent pseudomonads and pseudodistributive laws~\cite{marmolejo1997doctrines,marmolejo1999distributive,marmolejo2004distributive,marmolejo2012kan,walker2019distributive}; our main theorem will follow readily by combining several aspects of this theory.

\begin{definition}[{\cite[Theorem~10.7]{marmolejo1997doctrines}}]
    A pseudomonad $R$ on $\K$ is \emph{lax-idempotent}\footnotemark{} if, for every $R$-pseudoalgebra $(a, \alpha)$, the pseudoalgebra structure 1-cell $\alpha \colon R a \to a$ is left-adjoint to the unit component $\rho_a \colon a \to R a$ of $R$, and the counit of $\alpha \adj \rho_a$ is invertible.%
    \footnotetext{Our terminology is due to \cite{kelly1997property}.}
\end{definition}

\begin{definition}
    A pseudomonad $I$ is \emph{locally \ff{}} if its underlying pseudofunctor is locally \ff{}, \ie{} if each functor $I_{a, b} \colon \K(a, b) \to \K(Ia, Ib)$ is \ff{}.
\end{definition}

The three fundamental concepts with which we are concerned are cocompleteness, cocontinuity, and admissibility. We recall these notions now.

\begin{definition}
    \label{pseudoalgebra}
    Let $I$ and $R$ be lax-idempotent pseudomonads.
    \begin{enumerate}
        \item An object $a$ is \emph{$R$-cocomplete} if it is a $R$-pseudoalgebra, \ie{} if the unit component $\rho_a \colon a \to Ra$ of $R$ admits a left adjoint with invertible counit.
        \item A 1-cell $f \colon a \to b$ between $R$-cocomplete objects is \emph{$R$-cocontinuous} if it is a pseudomorphism of $R$-pseudoalgebras. This is the case if the canonical 2-cell below (which is always a lax morphism of $R$-pseudoalgebras), defined as the mate of the invertible 2-cell in the pseudonaturality square for $\rho$, is itself invertible~\cite[\S6]{marmolejo1997doctrines}.
        \[\begin{tikzcd}[sep=large]
        	& Ra & Rb & b \\
        	Ra & a & b
        	\arrow["{\rho_a}"{description}, hook, from=2-2, to=1-2]
        	\arrow["{\rho_b}"{description}, hook, from=2-3, to=1-3]
        	\arrow["f"', from=2-2, to=2-3]
        	\arrow["Rf", from=1-2, to=1-3]
        	\arrow["\alpha"', from=2-1, to=2-2]
        	\arrow["\beta", from=1-3, to=1-4]
        	\arrow["\iso"{description}, shorten <=8pt, shorten >=8pt, Rightarrow, from=1-2, to=2-3]
        	\arrow[""{name=0, anchor=center, inner sep=0}, Rightarrow, no head, from=2-1, to=1-2]
        	\arrow[""{name=1, anchor=center, inner sep=0}, Rightarrow, no head, from=2-3, to=1-4]
        	\arrow["{\varepsilon^b}", shorten <=2pt, shorten >=4pt, Rightarrow, from=1-3, to=1]
        	\arrow["{\eta^a}"', shorten <=4pt, shorten >=2pt, Rightarrow, from=0, to=2-2]
        \end{tikzcd}\]
        \item A 1-cell $f \colon a \to b$ is \emph{$I$-admissible} if $I f \colon I a \to I b$ is left-adjoint~\cite[Definition~1.1]{bunge1999bicomma}. In this case, we denote the right adjoint by $f^I$.
        \qedhere
    \end{enumerate}
\end{definition}

Observe in particular that $I$-admissibility is closed under composition, since $I$ is pseudofunctorial, and adjoints are closed under composition.

\begin{remark}
    \label{admissibility-as-relative-adjointness}
    In the enriched setting, admissibility with respect to a lax-idempotent pseudomonad may be seen to correspond to relative adjointness in a sharp sense, justifying our informal practice of speaking of $I$-admissibility as \emph{adjointness relative to $I$}. In particular, for $\Phi$ a class of weights, and denoting by $\overline\Phi$ the lax-idempotent pseudomonad exhibiting the free cocompletion under $\Phi$-weighted colimits (\cf{}~\cite{kelly2000monadicity}), a $\V$-functor $f \colon a \to b$ is $\overline\Phi$-admissible if and only if $f$ is left-adjoint relative to the cocompletion $a \to \Phi a$ of $a$ under $\Phi$-weighted colimits~\cite[Proposition~4.6]{lack2023accessible}.
    \[\begin{tikzcd}
    	& b \\
    	a && {\Phi a}
    	\arrow[""{name=0, anchor=center, inner sep=0}, "f", from=2-1, to=1-2]
    	\arrow[""{name=1, anchor=center, inner sep=0}, dashed, from=1-2, to=2-3]
    	\arrow[hook, from=2-1, to=2-3]
    	\arrow["\dashv"{anchor=center}, shift right=2, draw=none, from=0, to=1]
    \end{tikzcd}\]
    (While we should expect a similar statement also to hold in a formal setting, the notion of \emph{relative adjunction} requires a richer setting than 2-categories~\cite[Remark~5.9]{arkor2024formal}, so we shall not consider it here.)
\end{remark}

There are many equivalent formulations of the notion of pseudodistributive law $RI \tto IR$ between pseudomonads~\cite{walker2023no}. For us, the most appropriate will be the following.

\begin{definition}[{\cite[Theorem~35]{walker2019distributive}}]
    \label{composite}
    Let $R$ and $I$ be lax-idempotent pseudomonads on 2-category $\K$. $R$ \emph{distributes over} $I$ if there exists a lifting $\I$ of $I$ to the 2-category $R\h\Alg$ of $R$-pseudoalgebras, \ie{} a pseudomonad $\I$ on $R\h\Alg$ rendering the following square commutative, such that the pseudomonad structure of $\I$ is compatible with the pseudomonad structure of $I$ (\cf{}~\cite[Definition~7.3]{tanaka2004pseudo}).
    \[\begin{tikzcd}
    	R\h\Alg & R\h\Alg \\
    	\K & \K
    	\arrow["\I", from=1-1, to=1-2]
    	\arrow["I"', from=2-1, to=2-2]
    	\arrow["{U_R}"', from=1-1, to=2-1]
    	\arrow["{U_R}", from=1-2, to=2-2]
    \end{tikzcd}\]
\end{definition}

Note that, if a lifting $\I$ exists, then it is essentially unique and lax-idempotent~\cite[Corollaries~49 \& 50]{walker2019distributive}. Furthermore, the essentially unique composite pseudomonad $IR$ is also lax-idempotent~\cite[Theorem~11.7]{marmolejo1999distributive}.

\begin{example}
    \label{Rex-Ind}
    A motivating example, which provides a mnemonic for the symbols $I$ and $R$, is $\K \defeq \CAT$ the 2-category of locally small categories, $R \defeq \b{Rex}$ the free finite cocompletion pseudomonad, and $I \defeq \b{Ind}$ the free small filtered cocompletion pseudomonad, whose composite $IR$ is free small cocompletion pseudomonad. The existence of a lifting follows from the theory of sound limit doctrines~\cite{adamek2002classification} (\cf{}~\cite[Remark~6.6]{day2007limits}).
\end{example}

Before we can proceed with the proof of the main theorem, we require a small lemma regarding the composition of adjoints. It is an easy exercise to show that left adjoints in a 2-category are closed under composition. However, perhaps counterintuitively, it is possible under certain circumstances to obtain a left adjoint by composing a right adjoint with a left adjoint.

\begin{lemma}
    \label{reflection-compose-adjunction}
    Consider a pair of adjunctions in a 2-category $\K$, with unit and counit $\eta$ and $\varepsilon$, and $\eta'$ and $\varepsilon'$ respectively,
    \[\begin{tikzcd}
    	x & y & z
    	\arrow[""{name=0, anchor=center, inner sep=0}, "\ell"', shift right=2, from=1-2, to=1-1]
    	\arrow[""{name=1, anchor=center, inner sep=0}, "r"', shift right=2, hook, from=1-1, to=1-2]
    	\arrow[""{name=2, anchor=center, inner sep=0}, "{\ell'}", shift left=2, from=1-2, to=1-3]
    	\arrow[""{name=3, anchor=center, inner sep=0}, "{r'}", shift left=2, from=1-3, to=1-2]
    	\arrow["\dashv"{anchor=center, rotate=-90}, draw=none, from=0, to=1]
    	\arrow["\dashv"{anchor=center, rotate=-90}, draw=none, from=2, to=3]
    \end{tikzcd}\]
    for which $\varepsilon \colon \ell r \tto 1$ and $\ell' \eta r' \colon \ell' r' \tto \ell' r \ell r'$ are invertible. Then $\ell' r \adj \ell r'$.
\end{lemma}

This may essentially be seen as an adjoint lifting theorem along the lines of \cites[Theorem~3]{barr1972point}[Theorem~1.4]{power1988unified}.  We note that the proof is essentially contained in the proof of \cite[Lemma~2.1 \& Remark~2.2]{marmolejo2019level}, though the authors work under the additional assumption that $\varepsilon'$ is invertible. To reassure the reader that this assumption is not necessary, we provide a self-contained proof.

\begin{proof}
    We define the unit and counit by the following 2-cells.
    \begin{align*}
        & 1 \xtto{\varepsilon\inv} \ell r \xtto{\ell \eta' r} \ell r' \ell' r &
        & \ell' r \ell r' \xtto{(\ell' \eta r')\inv} \ell' r' \xtto{\varepsilon'} 1
    \end{align*}
    The triangle identities for $\ell \adj r$ and $\ell' \adj r'$ imply commutativity of the following diagrams, and hence the triangle identities for $\ell' r \adj \ell r'$.
    \[
    \begin{tikzcd}
    	{\ell' r} & {\ell' r \ell r} & {\ell' r \ell r' \ell' r} \\
    	& {\ell' r} & {\ell' r' \ell' r} \\
    	&& {\ell' r}
    	\arrow["{\ell' r \varepsilon\inv}", from=1-1, to=1-2]
    	\arrow["{\ell' r \ell \eta' r}", from=1-2, to=1-3]
    	\arrow["{(\ell' \eta r')\inv \ell' r}", from=1-3, to=2-3]
    	\arrow["{\varepsilon' \ell' r}", from=2-3, to=3-3]
    	\arrow["{\ell' \eta r}"{description}, from=2-2, to=1-2]
    	\arrow[Rightarrow, no head, from=2-2, to=1-1]
    	\arrow["{\ell' \eta' r}"{description}, from=2-2, to=2-3]
    	\arrow[Rightarrow, no head, from=2-2, to=3-3]
    \end{tikzcd}
    \hspace{4em}
    \begin{tikzcd}
    	{\ell r'} & {\ell r \ell r'} & {\ell r' \ell' r \ell r'} \\
    	& {\ell r'} & {\ell r' \ell' r'} \\
    	&& {\ell r'}
    	\arrow["{\varepsilon\inv \ell r'}", from=1-1, to=1-2]
    	\arrow["{\ell \eta' r \ell r'}", from=1-2, to=1-3]
    	\arrow["{\ell r' (\ell' \eta r')\inv}", from=1-3, to=2-3]
    	\arrow["{\ell r'\varepsilon'}", from=2-3, to=3-3]
    	\arrow["{\ell \eta r'}"{description}, from=2-2, to=1-2]
    	\arrow[Rightarrow, no head, from=2-2, to=1-1]
    	\arrow["{\ell \eta' r'}"{description}, from=2-2, to=2-3]
    	\arrow[Rightarrow, no head, from=2-2, to=3-3]
    \end{tikzcd}
    \]
\end{proof}

A pseudodistributive law of lax-idempotent pseudomonads produces two new lax-idempotent pseudomonads: the lifted pseudomonad $\I$ on $R\h\Alg$, and the composite pseudomonad $IR$ on $\K$. We are interested in characterising cocompleteness, cocontinuity, and admissibility for both: it is admissibility that is the focus of this section, but cocompleteness and cocontinuity will be useful in its study.

We may characterise cocompleteness and cocontinuity for both $\I$ and $IR$ in terms of cocompleteness for $R$ and $I$.

\begin{lemma}
    \label{IR-cocomplete}
    Let $R$ and $I$ be lax-idempotent pseudomonads on a 2-category $\K$ for which $R$ distributes over $I$. The following are equivalent for an object $a$ of $\K$.
    \begin{enumerate}
        \item $a$ is $IR$-cocomplete.
        \item $a$ is $R$-cocomplete, and $\I$-cocomplete as an object of $R\h\Alg$.
        \item $a$ is $R$-cocomplete and $I$-cocomplete.
    \end{enumerate}
    Furthermore, the following are equivalent for a 1-cell $f \colon a \to b$ between $IR$-cocomplete objects.
    \begin{enumerate}
        \item[(1$\,'$)] $f$ is $IR$-cocontinuous.
        \item[(2$\,'$)] $f$ is $R$-cocontinuous, and $\I$-cocontinuous as a 1-cell of $R\h\Alg$.
        \item[(3$\,'$)] $f$ is $R$-cocontinuous and $I$-cocontinuous.
    \end{enumerate}
\end{lemma}

\begin{proof}
    The equivalence of (1) and (2), and (1$'$) and (2$'$) is \cite[Theorem~6.6]{marmolejo2004distributive}.
    
    An explicit description of a pseudoalgebra and pseudomorphism for a lifted pseudomonad is given in \cite[\S4]{marmolejo2004distributive}. In particular, an $\I$-pseudoalgebra is an object with both $R$-pseudoalgebra and $I$-pseudoalgebra structure; similarly, an $\I$-pseudomorphism is a 1-cell with both $R$-pseudomorphism and $I$-pseudomorphism structure. It is therefore immediate that (2) $\implies$ (3), and (2$'$) $\implies$ (3$'$).

    Conversely, suppose an object $a$ is $R$-cocomplete and $I$-cocomplete. Denote by $\alpha \colon R a \to a$ the $R$-pseudoalgebra structure of $a$. In the terminology of \cite[Definition~31]{walker2019distributive}, $\alpha$ is \emph{$R_I$-cocontinuous} since $\alpha$ is left-adjoint and therefore preserves left extensions. Therefore, \cite[Proposition~53(a)]{walker2019distributive} implies that $a$ is $\I$-cocomplete. Similarly, suppose a 1-cell between $IR$-cocomplete objects is $R$-cocontinuous and $I$-cocontinuous. Then \cite[Proposition~53]{walker2019distributive} implies that $f$ is $\I$-cocontinuous. Hence (3) $\implies$ (2), and (3$'$) $\implies$ (2$'$).
\end{proof}

\begin{remark}
    \label{remark-composite-is-coproduct}
    We observe in passing that \cref{IR-cocomplete} implies that the composite lax-idempotent pseudomonad $IR$ is given by the coproduct $R + I$ of its components. However, since such a characterisation is tangential to our main theorem, we defer a proof to the appendix (\cref{composite-is-coproduct}).
\end{remark}

We may characterise admissibility for both $\I$ (\cref{admissibility-for-lifting}) and $IR$ (\cref{admissibility-for-composite}) in terms of admissibility for $I$. We shall not be concerned with $\I$-admissibility hereafter, mentioning it only for completeness. However, the characterisation of $IR$-admissibility will provide one direction of our general adjoint functor theorem.

\begin{proposition}
    \label{admissibility-for-lifting}
    Let $R$ and $I$ be lax-idempotent pseudomonads on a 2-category $\K$ for which $R$ distributes over $I$, and let $f \colon a \to b$ be an $R$-cocontinuous 1-cell between $R$-cocomplete objects. The following are equivalent.
    \begin{enumerate}
        \item $f$ is $\I$-admissible as a 1-cell in $R\h\Alg$.
        \item $f$ is $I$-admissible as a 1-cell in $\K$ and $f^I$ is $IR$-cocontinuous.
    \end{enumerate}
\end{proposition}

\begin{proof}
    Since $R\h\Alg \to \K$ is locally \ff{}, adjunctions in $R\h\Alg$ are precisely adjunctions in $\K$ with $R$-cocontinuous right adjoints. Furthermore, by \cref{IR-cocomplete}, a 1-cell is $IR$-cocontinuous if and only if it is $R$-cocontinuous and $I$-cocontinuous. The result then follows from the characterisation of admissibility for lifted lax-idempotent pseudomonads in \cite[Proposition~54]{walker2019distributive}.
\end{proof}

\begin{lemma}
    \label{admissibility-for-composite}
    Let $R$ and $I$ be lax-idempotent pseudomonads for which $R$ distributes over $I$, and let $f \colon a \to b$ be an $R$-cocontinuous 1-cell between $R$-cocomplete objects. The following are equivalent.
    \begin{enumerate}
        \item $f$ is $IR$-admissible.
        \item $f$ is $I$-admissible.
    \end{enumerate}
\end{lemma}

\begin{proof}
    (1) $\implies$ (2). Assume $f$ is $IR$-admissible. We have the following situation, denoting by $\rho \colon 1 \tto R$ the unit of $R$, and by $\alpha \colon R a \to a$ and $\beta \colon R b \to b$ the $R$-pseudoalgebra structures of $a$ and $b$ respectively.
    \[\begin{tikzcd}[sep=huge]
    	IRa & IRb \\
    	Ia & Ib
    	\arrow["If"', from=2-1, to=2-2]
    	\arrow[""{name=0, anchor=center, inner sep=0}, "IRf", shift left=2, from=1-1, to=1-2]
    	\arrow[""{name=1, anchor=center, inner sep=0}, "{I\rho_a}"', shift right=2, hook, from=2-1, to=1-1]
    	\arrow[""{name=2, anchor=center, inner sep=0}, "{I\rho_b}"', shift right=2, hook, from=2-2, to=1-2]
    	\arrow[""{name=3, anchor=center, inner sep=0}, "I\alpha"', shift right=2, from=1-1, to=2-1]
    	\arrow[""{name=4, anchor=center, inner sep=0}, "I\beta"', shift right=2, from=1-2, to=2-2]
    	\arrow[""{name=5, anchor=center, inner sep=0}, "{f^{IR}}", shift left=2, from=1-2, to=1-1]
    	\arrow["\dashv"{anchor=center}, draw=none, from=3, to=1]
    	\arrow["\dashv"{anchor=center}, draw=none, from=4, to=2]
    	\arrow["\dashv"{anchor=center, rotate=-90}, draw=none, from=0, to=5]
    \end{tikzcd}\]
    Taking
    \[\ell \defeq I \alpha \qquad r \defeq I \rho_a \qquad \ell' \defeq I \beta \circ I R f \qquad r' \defeq f^{IR} \circ I \rho_b\]
    in \cref{reflection-compose-adjunction}, using pseudonaturality of $\rho$, it therefore follows that $I f \adj I \alpha \circ f^{IR} \circ I \rho_b$.

    (2) $\implies$ (1). If $f$ is $I$-admissible, then $R f$ is $I$-admissible~\cite[Proposition~42]{walker2019distributive}, so that $IR f$ is left-adjoint.
\end{proof}

Any left-adjoint 1-cell is automatically cocontinuous for any lax-idempotent pseudomonad~\cite[Proposition~5.1]{marmolejo1997doctrines}; this aligns with the intuition that lax-idempotent pseudomonads are cocompletions. In fact, this is an instance of a more general phenomenon arising in the presence of a pseudodistributive law of lax-idempotent pseudomonads, as we show in \cref{admissibility-implies-cocontinuity}. First, we shall need a preparatory lemma.

\begin{lemma}
    \label{R-cocontinuous-iff-IR-cocontinuous}
    Let $R$ and $I$ be lax-idempotent pseudomonads for which $R$ distributes over $I$, and let $f \colon a \to b$ be a 1-cell between $R$-cocomplete objects. Suppose that $I$ is locally \ff{}. The following are equivalent.
    \begin{enumerate}
        \item $f$ is $R$-cocontinuous.
        \item $I f$ is $IR$-cocontinuous.
    \end{enumerate}
\end{lemma}

\begin{proof}
    First, note that, since $R$ is lax-idempotent, $f$ is canonically a lax morphism of $R$-pseudoalgebras (\cf{}~\cref{pseudoalgebra}).
    The result then follows directly from \cite[Proposition~54 \& Remark~56]{walker2019distributive}, which states that, under the assumption of local \ffness{} of $I$, $f$ is an $R$-pseudomorphism (\ie{} $f$ is $R$-cocontinuous) if and only if $I f$ is an $I R$-pseudomorphism (\ie{} $I f$ is $IR$-cocontinuous).
\end{proof}

\begin{lemma}
    \label{admissibility-implies-cocontinuity}
    Let $R$ and $I$ be lax-idempotent pseudomonads for which $R$ distributes over $I$, and let $f \colon a \to b$ be a 1-cell between $R$-cocomplete objects. Suppose that $I$ is locally \ff{}. If $f$ is $I$-admissible, then it is $R$-cocontinuous.
\end{lemma}

\begin{proof}
    Since $f$ is $I$-admissible, $I f$ is left-adjoint by definition. Thus $I f$ is $IR$-cocontinuous, from which the result follows by \cref{R-cocontinuous-iff-IR-cocontinuous}.
\end{proof}

\begin{remark}
    Denote by $\iota$ the unit of the lax-idempotent pseudomonad $I$ above. From the perspective of relative adjunctions (\cf{}~\cref{admissibility-as-relative-adjointness}), \cref{admissibility-implies-cocontinuity} is a reformulation of the fact that a left $(\iota_a \colon a \to I a)$-relative adjoint preserves those colimits that are preserved by $\iota_a$~\cite[Proposition~5.11]{arkor2024formal}, since the existence of the pseudodistributive law implies that $\iota_a$ is $R$-cocontinuous.
\end{remark}

\begin{example}
    \label{tholen}
    \Cref{admissibility-implies-cocontinuity} specialises to \cite[Theorem~2.2]{tholen1984pro} when $\K = \CAT$, taking $I$ to be the free cocompletion under colimits indexed by a class of small categories, containing the terminal category.
\end{example}

We may now give our main theorem, which is a generalised adjoint 1-cell theorem.

\begin{theorem}
    \label{DAFT}
    Let $R$ and $I$ be lax-idempotent pseudomonads for which $R$ distributes over $I$. Suppose that $I$ is locally \ff{}. A 1-cell $f \colon a \to b$ between $R$-cocomplete objects is $I$-admissible if and only if it is $R$-cocontinuous and $IR$-admissible.
\end{theorem}

\begin{proof}
    Sufficiency follows from \cref{admissibility-for-composite}. For necessity, suppose that $f$ is $I$-admissible. Then, by \cref{admissibility-implies-cocontinuity}, it is $R$-cocontinuous. Therefore, by \cref{admissibility-for-composite}, $f$ is $IR$-admissible.
\end{proof}

In light of \cref{admissibility-as-relative-adjointness}, \cref{DAFT} may be viewed as a \emph{relative adjoint functor theorem}. We obtain a formal adjoint 1-cell theorem as an immediate corollary.

\begin{corollary}
    \label{AFT}
    Let $R$ be a lax-idempotent pseudomonad. A 1-cell between $R$-cocomplete objects is left-adjoint if and only if it is $R$-cocontinuous and $R$-admissible.
\end{corollary}

\begin{proof}
    Take $I$ to be the identity pseudomonad in \cref{DAFT}.
\end{proof}

\begin{remark}
    \Cref{AFT} is analogous to the formal adjoint 1-cell theorems of \cites[Proposition~21]{street1978yoneda}[Corollary~10]{wood1982abstract}. The theorem of \citeauthor{street1978yoneda}, and that of \citeauthor{wood1982abstract}, is more general than \cref{AFT} in one respect, in that they do not require the existence of all colimits in the domain of the 1-cell (\cf{}~\cref{not-all-colimits}). However, in contrast to their formulations, our formulation directly specialises to adjoint functor theorems of interest, and is consequently often more practical.
\end{remark}

\begin{remark}
    \label{envelope}
    A \emph{(Gabriel--Ulmer) envelope} in the sense of \cite[Definition~3.1]{di2023accessibility} is, in particular, a pseudodistributive law of lax-idempotent pseudomonads $RI \tto IR$ for which $I$ and $R$ are locally \ff{} (\cf{}~\cite[Remark~1.12]{bunge1999bicomma}). Every envelope thus satisfies the assumptions of \cref{DAFT}.
\end{remark}

\begin{example}
    \label{adjoint-iff-cocontinuous-plus-admissible}
    \Cref{AFT} specialises to \cite[Theorem~2.3]{tholen1984pro} under the same assumptions as \cref{tholen}.
\end{example}

\begin{example}
    \label{well-cocomplete-AFT}
    Take $\K = \CAT$ and $R$ to be the free cocompletion under small colimits and large cointersections of regular epimorphisms. Then \cref{AFT} asserts that a functor between well-cocomplete categories, in the sense of \cite[Remark~4.39]{lack2023virtual}, is left-adjoint if and only if it is small-cocontinuous, preserves large cointersections of regular epimorphisms, and satisfies the solution set condition. While these assumptions are marginally stronger than those of \citeauthor{freyd1964abelian}'s adjoint functor theorem, they are typically satisfied in examples of interest (\cf{}~\cite{kelly1981large}).
\end{example}

\begin{example}
    \label{RAFT-for-LP-categories}
    Let $A$ and $B$ be locally presentable categories. In particular, $A$ and $B$ are small-complete~\cite[Corollary~1.28]{adamek1994locally}. The dual of \cref{AFT}, with respect to the free completion under small limits, states that a functor $f \colon A \to B$ is right-adjoint if and only if it is small-continuous and \emph{small-coadmissible}, in that, for every object $b \in B$, the copresheaf $B(b, f{-}) \colon A \to \Set$ is \emph{small}, \ie{} a small limit of corepresentables. In the terminology of \cite[Definition~4.11]{lack2023virtual}, a functor is small-coadmissible if and only if it admits a \emph{virtual left adjoint}. Consequently, by \cite[Proposition~4.20]{lack2023virtual}, $f$ is small-coadmissible if and only if it is accessible. We immediately recover the fact that a functor between locally presentable categories is right-adjoint if and only if it is small-continuous and accessible~(\cite[Theorem~1.66]{adamek1994locally}).
\end{example}

\begin{corollary}
    \label{Phi-AFT}
    Let $\V$ be a complete and cocomplete symmetric closed monoidal category and let $\Psi$ and $\Phi$ be classes of small weights for which
    \begin{enumerate}
        \item the free $\Phi$-cocompletion $\Phi a$ of any $\Psi$-cocomplete $\V$-category $a$ is $\Psi$-cocomplete;
        \item the functor $\Phi f \colon \Phi a \to \Phi b$ induced by any $\Psi$-cocontinuous $\V$-functor $f \colon a \to b$ between $\Psi$-cocomplete $\V$-categories is $\Psi$-cocontinuous.
    \end{enumerate}
    Then a functor between $\Psi$-cocomplete locally small categories is $\Phi$-admissible if and only if it is $(\Phi \cup \Psi)$-admissible and $\Psi$-cocontinuous.
\end{corollary}

\begin{proof}
    Under the given assumptions on $\V$, every class $\Psi$ of small weights induces a \lff{} lax-idempotent pseudomonad $\overline\Psi$ on $\VCAT$ whose pseudoalgebras are the $\V$-categories admitting $\Psi$-weighted colimits and whose pseudomorphisms are the $\V$-functors preserving $\Psi$-weighted colimits (\cf~\cite{kelly2000monadicity}). For classes $\Psi$ and $\Phi$ of small weights, conditions (1) and (2) simply express the condition that $\overline\Psi$ distributes over $\overline\Phi$.
    Observe that $\overline{\Phi \cup \Psi} \equiv \overline\Phi + \overline\Psi$, since free $(\Phi \cup \Psi)$-cocompletions are given by closing under $(\Phi \cup \Psi)$-weighted colimits of representables, hence by closing under $\Phi$- and $\Psi$-weighted colimits of representables. Thus, by \cref{composite-is-coproduct}, we have that $(\overline\Phi \c \overline\Psi)$-admissibility coincides with $(\overline\Phi + \overline\Psi)$-admissibility, which is equivalently $(\Phi \cup \Psi)$-admissibility. The result thus follows from \cref{DAFT}.
\end{proof}

\begin{example}
    \label{AFTs-for-weights}
    Let $\V$ be a complete and cocomplete symmetric closed monoidal category, and let $\Psi$ and $\Phi$ be classes of weights, inducing locally \ff{} lax-idempotent pseudomonads $\overline\Psi$ and $\overline\Phi$ on $\VCAT$. That $\overline\Psi$ distributes over $\overline\Phi$ to form the small-cocompletion pseudomonad is equivalent to asking for $\Phi$ to be a \emph{companion} for $\Psi$ in the terminology of \cite{lack2023accessible} (\cf{}~Proposition~4.9 \ibid{}). Each such pair is thus an envelope in the sense of \cref{envelope}. A notable class of examples is given by taking $\Phi \defeq \Psi^{{+}}$ the class of $\Psi$-flat weights when $\Psi$ is \emph{weakly sound} in the sense of \cite[Definition~3.2]{lack2023virtual}: various examples are given in Example 4.8 \ibid.

    Below we list several examples when $\V = \Set$. For each of the following, \cref{Phi-AFT} asserts that a functor between locally small $\Psi$-cocomplete categories is left-$\Phi$-adjoint if and only if it is small-admissible and $\Psi$-cocontinuous.
    \begin{center}
    \begin{tblr}{cccc}
      $\Psi$ & $\Phi$ & $\Phi$-adjoints & Reference \\
      \hline
      Small weights & $\varnothing$ & Adjoints & \cite[(13)]{ulmer1971adjoint} \\
      Small weights & Absolute weights & Semiadjoints & --- \\
      Finite weights & Filtered weights & Pluriadjoints & \cite[Theorem~2.1']{solian1990pluri} \\
      Connected weights & Discrete weights & Multiadjoints & \cite[Th\'eor\`eme~3.0.4]{diers1977categories} \\
      $\varnothing$ & Small weights & Virtual adjoints~\cite{lack2023virtual} & \cf{}~\cite[Axiom~2]{street1978yoneda}
    \end{tblr}
    \end{center}
\end{example}

\begin{example}
    \label{relative-AFT}
    As a special case of \cref{AFTs-for-weights}, observe that every functor from a small category is small-admissible, in which case we have a particularly sharp correspondence between $\Phi$-adjointness and $\Psi$-cocontinuity. For instance, let $\Psi$ be the class of finite weights, let $\Phi \defeq \Psi^{{+}}$ be the class of filtered weights, and let $a$ be a small finitely complete $\V$-category. It then follows from \cref{admissibility-as-relative-adjointness} that a $\V$-functor $f \colon a \to b$ is left-adjoint relative to the cocompletion $a \ffto \b{Ind}(a)$ of $a$ under filtered colimits if and only if $f$ preserves finite colimits. \Cref{AFTs-for-weights} thus recovers the relative adjoint functor theorem of \cite[\S3.2]{arkor2022monadic}.
    \[\begin{tikzcd}
    	& b \\
    	a && {\b{Ind}(a)}
    	\arrow[""{name=0, anchor=center, inner sep=0}, dashed, from=1-2, to=2-3]
    	\arrow[""{name=0p, anchor=center, inner sep=0}, phantom, from=1-2, to=2-3, start anchor=center, end anchor=center]
    	\arrow[""{name=1, anchor=center, inner sep=0}, "f", from=2-1, to=1-2]
    	\arrow[""{name=1p, anchor=center, inner sep=0}, phantom, from=2-1, to=1-2, start anchor=center, end anchor=center]
    	\arrow[hook, from=2-1, to=2-3]
    	\arrow["\dashv"{anchor=center, rotate=2}, shift right=2, draw=none, from=1p, to=0p]
    \end{tikzcd}\]
\end{example}

\begin{remark}
    \label{more-general-AFTs}
    We do not know whether the weak adjoint functor theorem of \cite[Theorem~2.1]{kainen1971weak} follows from \cref{DAFT}. This would require finding a lax-idempotent pseudomonad $I$ on $\CAT$ such that $I$ lifts to the 2-category of locally small categories with small coproducts, and for which $I$-admissibility is equivalent to weak adjointness. There does not appear to be a clear candidate. In analogy with the lifting of the free sifted cocompletion to categories with \emph{finite} coproducts, one might expect that the free cocompletion under reflexive coequalisers would be a suitable candidate for $I$. Unfortunately, \cite[Example~4.3]{adamek2000duality} exhibits an obstruction (\cf{}~\cite[\S3]{adamek2000duality}). Another candidate for $I$ is the free cocompletion under cointersections of regular epimorphisms (\cf{}~\cref{well-cocomplete-AFT}). However, in this case, it is not clear that $I$ lifts as required.

    Similar remarks apply to the weak adjoint functor theorems of \cite[Theorem~7.7]{bourke2023adjoint}, and a hypothetical polyadjoint functor theorem in the sense of \cite[35]{lamarche1988modelling}.
\end{remark}

\begin{remark}
    \label{not-all-colimits}
    \Cref{DAFT} requires both the domain and the codomain of the $R$-cocontinuous 1-cell $f \colon a \to b$ to be $R$-cocomplete, since this is required to even speak of $R$-cocontinuity in terms of $R$-pseudomorphisms. However, strictly speaking, it ought not to be necessary to assume the existence of all $R$-colimits in the codomain $b$. It is likely our proof could be modified to require only the preservation by $f$ of certain left extensions (\cf{}~\cite[Definition~20]{walker2019distributive}). However, in practice, $R$-cocompleteness of both the domain and codomain is usually satisfied, and we view the simplicity of the proof of \cref{DAFT} to outweigh the marginal gain in generality offered by such a modification.
\end{remark}

\subsection{Adjointness and presentability}

We conclude with an application of \cref{admissibility-for-composite} to the theory of locally presentable categories, complementing that of \cref{RAFT-for-LP-categories}. This does not require the full strength of \cref{DAFT}, but is nonetheless another useful consequence of the interaction between pseudodistributive laws and admissibility in relation to adjointness.

For the remainder, we shall take $R$ and $I$ to be lax-idempotent pseudomonads for which $R$ distributes over $I$ (local full faithfulness of $I$ is not required), and denote by $\iota \colon 1 \tto I$ the unit of $I$. We recall from \cite[\S5]{marmolejo2012kan} that $I$-cocompleteness of an object $a$ may be characterised in terms of the existence of left extensions along the unit $\iota_a \colon a \to I a$.

\begin{lemma}
    \label{left-adjoint-left-extension}
    Let $a$ be an $R$-cocomplete object and let $x$ be an $IR$-cocomplete object. Let $f \colon a \to x$ be an $R$-cocontinuous and $IR$-admissible 1-cell. Then the left extension $\iota_a \lx f \colon I a \to x$ of $f$ along $\iota_a$ is left-adjoint.
\end{lemma}

\begin{proof}
    Since $I$ is lax-idempotent and $x$ is $IR$-cocomplete, hence in particular both $I$-cocomplete and $R$-cocomplete by \cref{IR-cocomplete}, the following diagram commutes up to isomorphism, where we denote by $\chi^I$ the $I$-pseudoalgebra structure of $x$.
    \[\begin{tikzcd}[sep=large]
    	Ia & Ix \\
    	a & x
    	\arrow["f"', from=2-1, to=2-2]
    	\arrow[""{name=0, anchor=center, inner sep=0}, "{\iota_a \rhd f}"{description}, from=1-1, to=2-2]
    	\arrow["{\iota_a}", from=2-1, to=1-1]
    	\arrow["If", from=1-1, to=1-2]
    	\arrow["{\chi^I}", from=1-2, to=2-2]
    	\arrow["\iso"{description, pos=0.7}, draw=none, from=0, to=1-2]
    	\arrow["\iso"{description, pos=0.3}, draw=none, from=2-1, to=0]
    \end{tikzcd}\]
    Since $f$ is $R$-cocontinuous and $IR$-admissible, it is $I$-admissible by \cref{admissibility-for-composite}. Therefore $I f$ and $\chi^I$, and hence also their (pseudo)composite $\iota_a \lx f$, admit right adjoints.
\end{proof}

\begin{definition}
    \label{presentable-object}
    An object $x$ of $\K$ is \emph{$(R, I)$-presentable} if there exists an $R$-cocomplete object $a$ for which $I a \equiv x$, and for which every 1-cell with domain $a$ is $IR$-admissible.
\end{definition}

\begin{remark}
    Taking $R$ and $I$ to be an envelope in the sense of \cref{envelope}, \cref{presentable-object} recovers the \emph{presentable objects} of \cite[Definition~2.40]{di2023accessibility}.
\end{remark}

\begin{corollary}
    \label{AFT-for-presentable-objects}
    A 1-cell from an $(R, I)$-presentable object into an $IR$-cocomplete object is left-adjoint if and only if it is $IR$-cocontinuous.
\end{corollary}

\begin{proof}
    Since $I$ is lax-idempotent, a $I$-cocontinuous 1-cell $g \colon I a \to x$ into an $I$-cocomplete object is a left extension $\iota_a \lx (g \iota_a)$ along the unit of $I$. Since $a$ is $R$-cocomplete and $g$ is $R$-cocontinuous, then $g \iota_a$ is also $R$-cocontinuous, since $\iota_a$ is $R$-cocontinuous by virtue of pseudodistributivity. Finally, $g \iota_a$ is $IR$-admissible by $(R, I)$-presentability of $I a$. The result then follows from \cref{left-adjoint-left-extension}.
\end{proof}

\begin{example}
    \label{presentable-AFT}
    Taking $I$ and $R$ as in \cref{Rex-Ind}, an $(R, I)$-presentable object is exactly a locally finitely presentable category. Consequently, \cref{AFT-for-presentable-objects} states that every functor from a locally finitely presentable category into a small-cocomplete category is left-adjoint if and only if it is small-cocontinuous~(\cf{}~\cref{RAFT-for-LP-categories}). Corresponding statements hold for presentability with respect to any weakly sound class of weights $\Psi$ (\cf{}~\cite{adamek2002classification,lack2011notions}).
\end{example}

\appendix

\section{Composite lax-idempotent pseudomonads are coproducts}

\newcommand{\GRAY}{\b{GRAY}}
\newcommand{\Psmnd}{\b{Psmd}}
\newcommand{\Lift}{\b{Lift}}
\newcommand{\PsmdK}{\Psmnd_\K}
\newcommand{\LiftK}{\Lift_\K}

Our purpose in this section is to give a proof of the observation in \cref{remark-composite-is-coproduct}.
In the following, we denote by $\GRAY$ the Gray-category of Gray-categories, by $\PsmdK$ the 2-category of pseudomonads and pseudomorphisms on $\K$, and by $\LiftK$ the 2-category of pseudomonads and liftings on $\K$. Explicitly, $\PsmdK\op$ is given by the sub-2-category of the Gray-category $\Psmnd(\GRAY)$ of pseudomonads and pseudomorphisms in $\GRAY$~\cite[\S2]{gambino2021formal} that is spanned by pseudomonads on $\K$, pseudomonad morphisms whose 1-cells are trivial, and pseudomonad transformations whose 2-cells are trivial; similarly, $\LiftK\op$ is given by the sub-2-category of the Gray-category $\Lift(\GRAY)$ of pseudomonads and liftings in $\GRAY$~\cite[\S3]{gambino2021formal} that is spanned by pseudomonads on $\K$, liftings whose 1-cells are trivial, and lifting transformations whose 2-cells are trivial.

\begin{theorem}
    \label{composite-is-coproduct}
    Let $R$ and $I$ be lax-idempotent pseudomonads on a 2-category $\K$ for which $R$ distributes over $I$. Then the composite lax-idempotent pseudomonad $IR$ is the bicoproduct $R + I$ in the 2-category $\PsmdK$.
\end{theorem}

\begin{proof}
    Specialising \cite[Theorem~3.4]{gambino2021formal}, which establishes a correspondence between pseudomonad morphisms and liftings, to $\PsmdK$, it follows that the 2-functor $\ph\h\Alg \colon \PsmdK\op \biequiv \b{Lift}_\K\op \ffto \GRAY/\K$ is essentially \ff{}, and hence reflects bilimits. By \cref{IR-cocomplete}, $IR\h\Alg$ exhibits the following pseudopullback of 2-categories, which is the pseudoproduct $U_R \times U_I$ in $\GRAY/\K$. Consequently, $IR$ is the bicoproduct $R + I$ in $\PsmdK$.
    \[\begin{tikzcd}
    	IR\h\Alg & I\h\Alg \\
    	R\h\Alg & \K
    	\arrow[from=1-1, to=1-2]
    	\arrow[from=1-1, to=2-1]
    	\arrow[""{name=0, anchor=center, inner sep=0}, "{U_R}"', from=2-1, to=2-2]
    	\arrow["{U_I}", from=1-2, to=2-2]
    	\arrow["\lrcorner"{anchor=center, pos=0.125}, draw=none, from=1-1, to=0]
    \end{tikzcd}\]
\end{proof}

\printbibliography

\end{document}